\documentclass[10pt,a4paper]{article}
\usepackage{a4wide}
\usepackage{amsfonts,amsmath,amssymb,amsthm}
\usepackage[T1]{fontenc}
\usepackage{verbatim}

\theoremstyle{plain}
\newtheorem{theorem}{Theorem}
\newtheorem{proposition}[theorem]{Proposition}

\newtheorem{lemma}[theorem]{Lemma}

\theoremstyle{definition}

\newtheorem{remark}[theorem]{Remark}

\title{Nil Clean Involutions}
\author{Janez \v{S}ter\\Faculty of Mechanical Engineering\\University of Ljubljana\\\scriptsize{janez.ster@fs.uni-lj.si}}
\date{December 7, 2015}

\begin{document}

\maketitle

\begin{abstract}
We prove that if an involution in a ring is the sum of an idempotent
and a nilpotent then the idempotent in this decomposition must be $1$.
As a consequence, we completely characterize weakly nil-clean rings
introduced recently in
[Breaz, Danchev and Zhou, Rings in which every element is either a sum or
a difference of a nilpotent and an idempotent, J.~Algebra Appl.,
DOI: 10.1142/S0219498816501486].
\end{abstract}

In this note rings are unital.
$U(R)$, $\operatorname{Id}(R)$, $\operatorname{Nil}(R)$
and $\operatorname{Nil}^*(R)$ stand for the set of units, the set of idempotents,
the set of nilpotents and the upper nilradical of a ring $R$, respectively.
$\mathbb{Z}_n$ stands for the set of integers modulo $n$.
An \textit{involution} in a ring means an element $a$ satisfying $a^2=1$.

Following \cite{diesl}, we say that an element in a ring is \textit{nil clean}
if it is the sum of an idempotent and a nilpotent,
and a ring is nil clean if every element is nil clean.
The main result in this note is the following:

\begin{proposition}
\label{glavna}
Let $R$ be a ring with an involution $a\in R$. If $a$ is the sum of an idempotent
$e$ and a nilpotent $q$ then $e=1$. In particular, every nil clean involution
in a ring is unipotent (i.e.~$1$ plus a nilpotent).
\end{proposition}

\begin{proof}
Write $a=e+q$ with $e\in\operatorname{Id}(R)$ and $q\in\operatorname{Nil}(R)$,
and denote $f=1-e\in\operatorname{Id}(R)$ and $r=q(1+q)\in\operatorname{Nil}(R)$.
From $fq=f(a-e)=fa$ we compute $fr=fq(1+q)=fa(1+a-e)=fa(f+a)=faf+fa^2=faf+f$,
and similarly $rf=faf+f$. Hence $fr=rf$, so that $r$ is a nilpotent which commutes
with $f$, $e$, $q$ and $a$. Accordingly,
$$f=fa^2=fqa=fr(1+q)^{-1}a=f(1+q)^{-1}a\cdot r$$
is a nilpotent and hence $f=0$, as desired.
\end{proof}

Following \cite{breazdanchevzhou}, we say that a ring is \textit{weakly nil-clean}
if every element is either a sum or a difference of a nilpotent and an idempotent.

\begin{lemma}
\label{gllema}
If $R$ is a weakly nil-clean ring with $2\in U(R)$ then
$R/\operatorname{Nil}^*(R)\cong\mathbb{Z}_3$.
\end{lemma}

\begin{proof}
Choose any idempotent $e\in\operatorname{Id}(R)$, and set $a=1-2e$.
By assumption, either $a$ or $-a$ is nil clean. If $a$ is nil clean then, since $a^2=1$,
Proposition \ref{glavna} gives that $a-1=-2e$ is a nilpotent, so that $e$ is a nilpotent
and hence $e=0$. Similarly, if $-a$ is nil clean then, since $(-a)^2=1$,
Proposition \ref{glavna} gives that $-a-1=-2(1-e)$ is a nilpotent, so that
$1-e$ is a nilpotent and hence $e=1$. This proves that $R$ has only trivial idempotents.
Accordingly, since $R$ is weakly nil clean, every element of $R$ must be either
$q$ or $1+q$ or $-1+q$ for some $q\in\operatorname{Nil}(R)$.
From this, one quickly obtains that $\operatorname{Nil}(R)$ must actually form
an ideal in $R$, so that $R/\operatorname{Nil}^*(R)$ can have only $3$ elements
and hence $R/\operatorname{Nil}^*(R)\cong\mathbb{Z}_3$, as desired.
(Alternatively, considering that $R$ is abelian,
$R/\operatorname{Nil}^*(R)\cong\mathbb{Z}_3$ can be also obtained
from \cite[Theorem 12]{breazdanchevzhou}.)
\end{proof}

Using the above lemma, we have:

\begin{theorem}
A ring is weakly nil clean if and only if it is either nil clean or isomorphic to
$R_1\times R_2$ where $R_1$ is nil clean and
$R_2/\operatorname{Nil}^*(R_2)\cong\mathbb{Z}_3$.
\end{theorem}

\begin{proof}
Follows from Lemma \ref{gllema} together with \cite[Theorem 5]{breazdanchevzhou}.
\end{proof}

\begin{remark}
Proposition \ref{glavna} can be generalized to arbitrary algebraic elements of order
$2$ as follows. Let $R$ be an algebra over a commutative
ring $k$, and let $a\in R$ be an element satisfying $\alpha a^2+\beta a+\gamma=0$,
with $\alpha,\beta,\gamma\in k$, and suppose that
$a=e+q$ with $e\in\operatorname{Id}(R)$ and $q^n=0$. Then one can show that
$r=q(\alpha q+\alpha+\beta)$ is a nilpotent commuting with $e$, which yields,
similarly as in Proposition \ref{glavna}, that
$$(\alpha+\beta)^ne+(\alpha+\beta)^{n-1}\gamma$$
is also a nilpotent. Note that this result indeed generalizes Proposition \ref{glavna}
(taking $\alpha=1$, $\beta=0$ and $\gamma=-1$ yields
that $e-1$ is a nilpotent, so that $e=1$).
However, for orders of algebraicity higher than $2$ this argument no longer seems
to work.
\end{remark}

\subsection*{Acknowledgements}

The author is indebted to Professor T.Y.~Lam for a helpful discussion on the previous
version of this work.

\bibliographystyle{abbrv}
\bibliography{References}

\end{document}